\RequirePackage{fix-cm}
\documentclass[oneside, graybox]{svmult}
\usepackage{algorithm}
\usepackage[noend]{algpseudocode}
\usepackage{graphicx}
\usepackage{amsmath}
\usepackage{amssymb}
\newcommand{\real}{\mathbb{R}}
\usepackage{enumerate}

\usepackage{type1cm}        
\usepackage{makeidx}         
\usepackage{graphicx}        
\usepackage{multicol}        
\usepackage[bottom]{footmisc}

\usepackage{newtxtext}       %
\usepackage[varvw]{newtxmath}       

\makeindex             

\DeclareMathOperator*{\argmin}{arg\,min}

\newtheorem{assumption}{Assumption}
\newcounter{algo}
{}

\begin{document}

\title*{On Implicit Concave Structures in Half-Quadratic Methods for Signal Reconstruction}

\author{Vittorio Latorre  \orcidID{0000-0002-4644-848X}   
}
\institute{Vittorio Latorre \at Department of Bioscience and Territory, University of Molise, c.da Fonte Lappone, snc, Pesche (IS), 86170, Molise, Italy, \email{vittorio.latorre@unimol.it}}

\maketitle

\abstract{
In this work, we introduce a new class of non-convex functions, called implicit concave functions, which are compositions of a concave function with a continuously differentiable mapping. We analyze the properties of their minimization by leveraging Fenchel conjugate theory to construct an augmented optimization problem. This reformulation yields a one-to-one correspondence between the stationary points and local minima of the original and augmented problems. Crucially, the augmented problem admits a natural variable splitting that reveals convexity with respect to at least one block, and, in some cases, leading to a biconvex structure that is more amenable to optimization. This enables the use of efficient block coordinate descent algorithms for solving otherwise non-convex problems. As a representative application, we show how this framework applies to half-quadratic regularization in signal reconstruction and image processing. We demonstrate that common edge-preserving regularizers fall within the proposed class, and that their corresponding augmented problems are biconvex and bounded from below. Our results offer both a theoretical foundation and a practical pathway for solving a broad class of structured non-convex problems.}

\section{Introduction}\label{Introduction}
Given a concave, continuous, and differentiable function $V:\mathbb{R}^m \rightarrow \overline{\mathbb{R}}$, where $\overline{\mathbb{R}} = [-\infty, \infty]$ denotes the extended real number line, and a continuous and differentiable function $\Phi:\mathbb{R}^n \rightarrow \mathbb{R}^m$, we are interested in studying the properties of the composition of $V$ with $\Phi$:
\begin{equation}\label{eq: icf}
f(x) = (V \circ \Phi)(x) = V(\Phi(x)).
\end{equation}
We refer to this class of functions as \emph{implicit concave functions}.

Our interest in this class of functions stems from the theoretical results presented in \cite{l22}, where a similar class—composed of a convex and continuous function—is introduced and used to formulate a dual problem analogous to the well-known Lagrangian dual in constrained optimization. Related theoretical developments are also proposed in \cite{snl19,ls14}, leading to the design of efficient numerical methods.
Although concavity is not as favorable as convexity from an optimization standpoint \cite{conc95}, it remains useful for modeling a wide range of engineering applications \cite{conc00,conc98} and can be exploited through tailored, efficient algorithms \cite{conc18}. 
Motivated by the promising results in \cite{l22}, we investigate the problem of minimizing an implicit concave function.

We show that the properties of implicit concave functions enable the formulation of an augmented function for which a one-to-one correspondence can be established between the stationary points of the original function and those of the augmented function. Furthermore, we prove that if a point is a local minimum of the original function—satisfying either the second-order necessary or sufficient optimality conditions—then the corresponding point in the augmented formulation is also a local minimum, satisfying the equivalent conditions. As a result, the minimizations of the original and augmented problems are practically equivalent, in the sense that the minimum of one function corresponds directly to the minimum of the other.
   
The augmented problem simplifies the minimization with respect to the original variables by introducing a new set of augmented variables, for which the problem is convex. This structure generally makes the minimization easier than the original formulation. 
In some cases, it is even possible to identify a subset of the augmented variable space where the augmented function is \emph{biconvex} \cite{bicon}, meaning that it is convex in the original variables when the augmented variables are fixed, and vice versa.
The natural partition of variables into two blocks—one consisting of the original variables and the other of the augmented variables—combined with the convexity of the augmented function with respect to the second block, makes this formulation particularly well-suited for block coordinate descent algorithms \cite{gs99}, such as the well-known nonlinear Gauss-Seidel method \cite{ber03}.

The properties of implicit concave functions have already been utilized in the literature—particularly in the context of half-quadratic regularization \cite{ge92,nik07,te98,hq21} for edge-preserving functions in signal reconstruction and image recognition \cite{ep_sur}. 
In these applications, the original signal or image is reconstructed from noise-corrupted data using edge-preserving functions \cite{te98}, which are designed to reduce noise without blurring or destroying important features such as edges or sharp transitions. These techniques have also been applied more recently in the context of deep learning for image recognition \cite{unr241,unr24}.
In this work, we briefly introduce the formulation of edge-preserving functions, demonstrate that they fall within the class of implicit concave functions, and show that the corresponding augmented problem is biconvex over its entire domain and bounded from below.

Our main contributions are therefore:
\begin{itemize}
\item An analysis of the minimization problem for implicit concave functions, and the introduction of an augmented problem whose minimization is proven to be equivalent to that of the original function;
\item An analysis of the half-quadratic regularization of edge-preserving functions, providing new insights into this problem and showing that the regularization is indeed biconvex and bounded.
\end{itemize}

The paper is organized as follows: In the next section, we present the general properties of implicit concave functions, define the associated augmented function, and prove the correspondence between the minimizers of the original and augmented functions. We also introduce a simple condition under which the augmented function is bounded from below. In Section 3, we apply implicit concave functions to edge-preserving regularization in signal processing and image recognition, showing that the augmented function is biconvex over its domain and bounded from below. Finally, in Section 4, we present our conclusions.


\section{Proprieties of Implicit Concave Functions}\label{prop}
\noindent
We start with the formal definition of an Implicit Concave Function:
\begin{definition}\label{def: start}
Let $V: \mathbb{R}^m \rightarrow \mathbb{R}$ be a continuous, differentiable, and stricly concave function, and let $\Phi: C \subseteq \mathbb{R}^n \rightarrow \mathbb{R}^m$ be a continuous and differentiable function. We say that $f: C \subseteq \mathbb{R}^n \rightarrow \mathbb{R}$ is an Implicit Concave Function if $f(x)$ can be expressed as:
$$
f(x) = (V \circ \Phi)(x) = V(\Phi(x)).
$$
\end{definition}
Given the presence of the concave function $V: \mathbb{R}^m \rightarrow \mathbb{R}$, we consider its conjugate function $V^*: \mathbb{R}^m \rightarrow \mathbb{R}$ defined as:
$$
V^*(\sigma)= \inf_{y\in\real^m} (\langle y,\sigma\rangle - V(y)).
$$
As $V$ is concave, also $V^*$ is concave, and it is possible to show that the Fenchel  inequality \cite{rock09} holds for implicit concave functions and their conjugates as follows:
\begin{theorem}\label{th: fen}
Let $f: C \rightarrow \overline{\mathbb{R}}$, with $f(x) = V(\Phi(x))$, be an implicit concave function according to Definition \ref{def: start}. Then, the following inequality always holds for any $x \in C$ and $\sigma \in D$:
\begin{equation}\label{eq: fen_in}
f(x) = V(\Phi(x)) \le \langle \Phi(x), \sigma \rangle - V^*(\sigma),
\end{equation}
where $V^*: \mathbb{R}^m \rightarrow \overline{\mathbb{R}}$ is the conjugate of $V$ and is therefore concave in $\sigma$, $\sigma \in D \subseteq \mathbb{R}^m$ is the vector of augmented variables, and $D := \mathcal{R}(\nabla V)$ is the space of augmented variables.
\end{theorem}
\begin{proof}
We consider the concave function $V: \mathbb{R}^m \rightarrow \mathbb{R}$ and its concave conjugate:
\begin{equation}\label{eq: fen_con}
V^*(\sigma) = \inf_{y} \left\{ \langle y, \sigma \rangle - V(y) \right\},
\end{equation}
where $y$ is a point in $\mathbb{R}^m$.
From the well-known Fenchel inequality in convex programming \cite{rock09}, we have:
\begin{equation}\label{eq: fen_in_nor}
V(y) \le \langle y, \sigma \rangle - V^*(\sigma), \quad \forall y \in \mathbb{R}^m, \, \forall \sigma \in D,
\end{equation}
which holds for any $y \in \mathbb{R}^m$ and $\sigma \in D$.

Since the inequality in (\ref{eq: fen_in_nor}) holds for any $y \in \mathbb{R}^m$, it also holds when $y$ is the result of applying the mapping $\Phi$ to a point $x \in C \subseteq \mathbb{R}^n$, i.e., when $y = \Phi(x)$. Therefore, for the function $f(x)$, the Fenchel inequality holds as follows:
$$
f(x) = V(\Phi(x)) \le \langle \Phi(x), \sigma \rangle - V^*(\sigma), \quad \forall x \in C, \, \forall \sigma \in D.
$$
This completes the proof of the theorem.
\end{proof}
We consider the optimization problem:
\begin{equation}\label{eq: prim}
\min_x f(x).
\end{equation}
From the relation in Theorem \ref{th: fen}, we can define the \emph{Augmented Function} associated with Problem (\ref{eq: prim}) as the right-hand side of (\ref{eq: fen_in}):
\begin{equation}\label{eq: aug}
L(x, \sigma) = \langle \Phi(x), \sigma \rangle - V^*(\sigma).
\end{equation}
It is easy to observe that $L$ is convex in the augmented variable $\sigma$ because $V^*$ is concave. Furthermore, the function $L$ has the property that there is a one-to-one correspondence between its stationary points and those of $f$, as stated in the following theorem:
\begin{theorem}\label{th: corr}
let $f(x)$ be an  implicit concave function and $L(x,\sigma)$ be its associated augmented function as in (\ref{eq: aug}). 
Then the  point $(\bar{x},\bar{\sigma})\in C\times D$ is a stationary point of $L(\cdot)$ if and only if $\bar{x}$, is a stationary point of $f(\cdot)$, 
and $\bar\sigma$ is the vector of augmented variables such that $\bar\sigma=\argmin_{\sigma\in D}\{\langle \Phi(\bar{x}),\sigma \rangle -V^*(\sigma) \}$,
 and
\begin{equation}\label{eq: valcorr}
f(\bar{x})=L(\bar{x},\bar{\sigma}).
\end{equation}
\end{theorem}

\begin{proof}
Let $(\bar{x},\bar{\sigma})$ be a stationary point of the Lagrangian function. 
By looking at its derivatives we obtain:
\begin{equation}\label{eq: ld}
\begin{array}{ccccc}
\nabla_x L(\bar{x},\bar{\sigma})&=&\nabla \Phi(\bar{x}) \bar{\sigma} &=&0\\
\nabla_\sigma L(\bar{x},\bar{\sigma})&=& \Phi(\bar{x})- \nabla V^*( \bar{\sigma}) &=&0\\
\end{array},
\end{equation}
From the properties of convex conjugate functions, we have $\nabla V(\cdot)^{-1}=\nabla V^*(\cdot)$, and the second equation of the (\ref{eq: ld}) becomes:
$$
\bar\sigma=\nabla V(\Phi(\bar x)),
$$
by plugging this value back into the first equation of the (\ref{eq: ld}), we obtain:
$$
\nabla \Phi(\bar{x}) \nabla V(\Phi(\bar{x}) )=0.
$$
by the fact that $f(x)=V(\Phi(x))$ and for the chain derivation rule, we have:
$$
\nabla f(\bar x)= \nabla \Phi(\bar{x}) \nabla V(\Phi(\bar{x}) )=0.
$$
That is $\bar{x}$ is a stationary point of $f(x)$.

For the sufficiency, let $\bar{x}$ be a stationary point of $f(x)$, and let $\bar{\sigma}$ be the point in the augmented space such that
$$
\bar{\sigma} = \arg \min_{\sigma \in D} \left\{ \langle \Phi(\bar{x}), \sigma \rangle - V^*(\sigma) \right\}.
$$
The point $\bar{\sigma}$ is well-defined because of the concavity properties of $V^*(\sigma)$, and it satisfies the first-order conditions of $L(\cdot)$ for the variable $\sigma$:
$$
\nabla_\sigma L(\bar{x}, \bar{\sigma}) = \Phi(\bar{x}) - \nabla V^*(\bar{\sigma}) = 0,
$$
and by the properties of the convex conjugate functions, we also have:
\begin{equation}\label{eq: temp}
\bar{\sigma} = \nabla V(\Phi(\bar{x})).
\end{equation}
The first-order conditions of $L(\cdot)$ for the variable $x$ are:
\begin{equation}\label{eq: lx}
\nabla_x L(x, \sigma) = \nabla \Phi(x) \sigma.
\end{equation}
By substituting (\ref{eq: temp}) into (\ref{eq: lx}), we get:
$$
\nabla_x L(\bar{x}, \bar{\sigma}) = \nabla \Phi(\bar{x}) \bar{\sigma} = \nabla \Phi(\bar{x}) \nabla V(\Phi(\bar{x})) = \nabla f(\bar{x}) = 0.
$$
Therefore, the point $(\bar{x}, \bar{\sigma})$ is a stationary point for $L(x, \sigma)$.

Finally, equation (\ref{eq: valcorr}) follows from the equality in (\ref{eq: fen_in}), which is satisfied when $\sigma$ is the minimum of $L(\bar{x}, \sigma)$ for any fixed $\bar{x}$:
$$
f(\bar{x}) = V(\Phi(\bar{x})) = \langle \Phi(\bar{x}), \bar{\sigma} \rangle - V^*(\bar{\sigma}) = L(\bar{x}, \bar{\sigma}).
$$
This proves the theorem.
\end{proof}

From Theorem~\ref{th: corr}, it follows that finding a stationary point of the augmented function \( L \) also yields a stationary point of the original function \( f \). However, this raises a natural question regarding the relationship between corresponding stationary points: in particular, we are interested in whether a stationary point of \( L \) corresponds to a local minimum of \( f \), and vice versa. 
Specifically, we analyze the second-order properties of the two functions. 
We show that if a stationary point \( \bar{x} \) of the original function \( f \) satisfies the second-order necessary conditions (SONC)—namely, the Hessian matrix is positive semidefinite—or the second-order sufficient conditions (SOSC)—namely, the Hessian is positive definite—for a local minimum, then the corresponding stationary point \( (\bar{x}, \bar{\sigma}) \) of the augmented function \( L \) also satisfies the respective SONC or SOSC.
This result is established by analyzing the second-order derivatives, as formalized in the following theorem.
\begin{theorem}\label{the: min}
Let $f(x)=V(\Phi(x))$ be an implicit concave function with $V$ twice differentiable, and let $L(x,\sigma)$ be its associated augmented function as in (\ref{eq: aug}). Then the point $\bar{x}$, with $\nabla^2 V(\Phi(\bar{x}))$ non-singular, is such that $\nabla^2 f(\bar{x})\succeq 0$ (resp. $\nabla^2 f(\bar{x})\succ 0$) if and only if its corresponding stationary point $(\bar{x},\bar{\sigma})$ is such that $\nabla^2 L(\bar{x},\bar{\sigma})\succeq 0$ (resp. $\nabla^2 L(\bar{x},\bar{\sigma})\succ 0$).
\end{theorem}
\begin{proof}
We report the proof for the case of a positive semidefinite Hessian matrix. The proof for the positive definite case follows the same reasoning.
The Hessians of $f$ and $L$ are:
\begin{subequations}\label{eq: h}
\begin{align}
\nabla^2 f(x) &= \nabla V(\Phi(x)) \nabla^2 \Phi(x)  +\nabla \Phi(x)^T \nabla^2 V(\Phi(x)) \nabla \Phi(x) \label{eq: hf},\\
\nabla^2 L(x,\sigma) &= 
\left(\begin{array}{ll}
\sigma \nabla^2 \Phi(x) & \nabla \Phi(x)\\
\nabla \Phi(x)^T & - \nabla^2 V^*(\sigma) 
\end{array}\right) \label{eq: hl}.
\end{align}
\end{subequations}
Let $\bar{x}$ be a stationary point of $f$ such that $\nabla^2 f(\bar{x}) \succeq 0$. For this point, we have that $- \nabla^2 V(\Phi(\bar{x}))^{-1} \succ 0$, since $V$ is concave. By the Schur complement, these two conditions are verified if and only if the following matrix is also positive semidefinite:
\begin{equation}\label{eq: schur1}
\left(\begin{array}{ll}
\nabla V(\Phi(\bar{x})) \nabla^2 \Phi(\bar{x}) & \nabla \Phi(\bar{x})\\
\nabla \Phi(\bar{x})^T & - \nabla^2 V(\Phi(\bar{x}))^{-1} 
\end{array}\right) \succeq 0.
\end{equation}
Since $(\bar{x}, \bar\sigma)$, with $\bar\sigma = \argmin_{\sigma \in D} \{ \langle \Phi(\bar{x}), \sigma \rangle - V^*(\sigma) \}$, is the stationary point of $L$ corresponding to $\bar{x}$, and by the properties of convex conjugate functions and their derivatives applied through the inverse function theorem \cite{gorni91}, we know that for the point $\bar{y} \in \real^m$, with $\bar{y} = \Phi(\bar{x})$, we have:
\[
\begin{array}{c}
\bar\sigma = \nabla V(\bar{y}) \vspace{0.5em} \\
\nabla^2 V^*(\bar\sigma) = \nabla^2 V(\bar{y})^{-1}
\end{array},
\]
therefore, (\ref{eq: schur1}) is equivalent to:
\[
\left(\begin{array}{ll}
\bar\sigma \nabla^2 \Phi(\bar{x}) & \nabla \Phi(\bar{x})\\
\nabla \Phi(\bar{x})^T & - \nabla^2 V^*(\bar\sigma) 
\end{array}\right) \succeq 0,
\]
which is exactly the Hessian of function $L$ in (\ref{eq: hl}) evaluated at the point $(\bar{x}, \bar\sigma)$. Therefore, $\nabla^2 f(\bar{x}) \succeq 0$ if and only if:
\[
\nabla^2 L(\bar{x}, \bar\sigma) \succeq 0.
\]
\end{proof}

Consequently, finding a point that satisfies the SONC or the SOSC for a minimum of $L$ is equivalent to finding a point that satisfies the same conditions for $f$. In other words, the problem of finding a minimum point of $f$ is equivalent to the problem of finding a minimum point of $L$, with the advantage that $L$ is convex in the block of augmented variables $\sigma$.
This structural property enables the use of block coordinate descent methods, which are particularly attractive in large-scale applications. In fact, convergence to a local minimum is guaranteed for some classes of block coordinate descent algorithms even when the overall problem is non-convex, as long as it involves only two variable blocks \cite{gs99}. 
Furthermore, it is easy to prove that if $f$ is bounded below, then $L$ is also bounded below:

\begin{theorem}\label{th: bound}
Let $f(x)$ be an implicit concave function, and 
suppose that $f(x)$ is bounded below; that is, there exists a scalar $\alpha$ such that $\alpha \le f(x)$ for all $x \in C$. Then its associated augmented function $L(x,\sigma)$ is also bounded below.
\end{theorem}
\begin{proof}
From inequality (\ref{eq: fen_in}) in Theorem \ref{th: fen} and the definition of $L$ in (\ref{eq: aug}), we have:
$$
\alpha \le f(x) = V(\Phi(x)) \le \langle \Phi(x),\sigma \rangle - V^*(\sigma) = L(x,\sigma), \quad \forall x \in C, \,\, \forall \sigma \in D.
$$
This proves the theorem.
\end{proof}

Therefore, if $f$ satisfies standard boundedness properties that ensure the existence of a minimum — for instance, if $f$ has compact level sets, then $L$ enjoys the same boundedness as well. This last theorem assures us that the problem of minimizing the augmented function $L$ is well posed, in the sense that $L$ is bounded below. This is a fundamental property in optimization, as it guarantees that a minimum exists and prevents the algorithm from diverging, making the reformulated problem both theoretically sound and practically solvable.

\section{Half Quadratic Regularization of Edge Preserving Functions in Image Recognition}\label{halfquad}
One of the most relevant domains where implicit concave functions naturally arise is in inverse problems involving noisy or incomplete data, specifically in signal processing, where the available data $b \in \real^r$ is obtained from an original unknown signal or image ${z}$. In many applications, this unknown data is related to the observed data through a linear system and some noise, as follows:
$$
b = A x + \eta,
$$
where the matrix $A \in \real^{r \times n}$ represents the linear transformation operated by the system, and $\eta$ is some Gaussian noise.

Due to the presence of noise, finding the minimum of the function $\| A z - b \|^2$ is an ill-posed problem, and regularization is necessary. Thus, the technique used to find a good approximation of the vector $z$ involves a data fidelity term and a regularization term, leading to the following optimization problem:
\begin{equation}\label{eq: primep}
\min_z f_{ep}(x) = \| A x - b \|^2 + \beta \sum_{i=1}^m \psi (\| G_i x \|),
\end{equation}
which consists of a quadratic data-fidelity term and a regularization term characterized by a regularization function $\psi: \real \to \real$. The argument of $\psi$ involves a series of linear operators $G_i : \real^n \to \real^s$ for $i = 1, \dots, m$ and $s \ge 1$, with $\beta > 0$ being a given regularization parameter.
In general \cite{det97}, the following assumptions are made to ensure the suitability of $\psi$ as a regularization function:
\begin{assumption}\label{ass: gen}
$\;$ \newline \vspace{-.3cm}
\begin{enumerate}
\item $\psi(t) \geq 0 \quad \forall t$ with $\psi(0) = 0$.
\item $\psi(t) = \psi(-t)$.
\item $\psi \in C^1$.
\end{enumerate}
\end{assumption}
Using a convex function $\psi$ in (\ref{eq: primep}) would guarantee the uniqueness of the minimum point \cite{cha97,ve01}, however, in several experiments, non-convex functions yield better results \cite{te98}. These improved results are due to properties gained by the potential function $\psi$ in exchange for convexity. One of these properties is the so-called edge-preserving regularization, which encourages smoothing within regions and discourages smoothing across boundaries, ensuring that in the reconstructed signal or image, the edges between homogeneous zones remain sharp and preserved.
In the literature \cite{dela98, ge92, nik07,te98}, edge-preserving functions are defined by the following assumptions, which ensure the desired edge-preserving regularization:
\begin{assumption}\label{ass:phi1}
$\;$ \newline \vspace{-.3cm}
\begin{enumerate}
\item $\psi'(t) \geq 0, \quad \forall t \geq 0$.
\item $\frac{\psi'(t)}{2t}$ is continuous and strictly decreasing on $[0, +\infty)$.
\item $\lim_{t \to +\infty} \frac{\psi'(t)}{2t} = 0$.
\item $\lim_{t \to 0^+} \frac{\psi'(t)}{2t} = M$, where $0 < M < +\infty$.
\end{enumerate}
\end{assumption}

While Assumption \ref{ass: gen} is for general regularization functions, Assumption \ref{ass:phi1} specifically applies to edge-preserving functions. It is possible to exploit this second set of assumptions to demonstrate that edge-preserving functions are, in fact, implicit concave.

\begin{theorem}\label{th: peic}
Let $\psi:\real\rightarrow \real_+$ be an edge preserving function satisfying Assumptions \ref{ass: gen} and   \ref{ass:phi1}. Then $\psi(t)$ is implicit concave, 
that is $\psi$ can be expressed as:
$$
\psi(t)=V(t^2)
$$
Where $V$ is  continuous, differentiable, and stricly concave and $\Phi=t^2$ is continuous and differentiable.
\end{theorem}
\begin{proof}
It is easy to notice that a consequence of Assumption \ref{ass: gen} and the second of Assumption  \ref{ass:phi1}  that $\psi(\sqrt{\tau})$ is concave. Therefore, it is possible to choose:
$$
V(y)=\psi(\sqrt{y}).
$$
If we set $y=t^2$, we have:
$$
\psi(t)=V(\Phi(t)), \quad\Phi(t)=t^2,
$$
where $V$ is not only concave but also continuous and differentiable for Assumption \ref{ass: gen} and \ref{ass:phi1},  and $\Phi(t)=t^2$ is continuous and differentiable.
\end{proof}

Therefore, in (\ref{eq: primep}), the function
$$
\psi(\|G_i z \|) = V(\Phi_i(z)), \quad \Phi_i(z) = \| G_i z \|^2,
$$
is implicit concave for every $i$. The augmented problem associated with Problem (\ref{eq: primep}) is:
\begin{equation}\label{eq: lep}
\min_{x,\sigma} L_{ep}(x,\sigma) = \|A x - b \|^2 + \beta \left( \sum_{i=1}^m \sigma_i \|G_i x \|^2 - V^*(\sigma_i) \right),
\end{equation}
where $V^*$ is the conjugate function of $V$. 
The first observation is that the function $L_{ep}$ is convex in $\sigma \in \mathbb{R}^m$ and quadratic in $x \in \mathbb{R}^n$. For this reason, Problem (\ref{eq: lep}) is known as a half-quadratic regularization, as introduced in the seminal paper \cite{ge92}. 

Having established that $\psi$ is implicit concave, all the results from Section \ref{prop} apply to Problem (\ref{eq: lep}). In particular, the minimization of $f_{ep}$ and $L_{ep}$ are equivalent, as a result of Theorems \ref{th: corr} and \ref{the: min}. Furthermore, Problem (\ref{eq: primep}) is bounded below, as a consequence of Assumption \ref{ass:phi1}, and therefore Problem (\ref{eq: lep}) is also bounded below, as a result of Theorem \ref{th: bound}.

Another interesting characteristic of half-quadratic functions comes from examining the subset $D^+ \subseteq D$ of the augmented space where the function $L_{ep}$ is convex in the variables $x$. By inspecting Problem (\ref{eq: lep}), we observe that $L_{ep}$ is convex in $x$ if the following matrix is positive semidefinite:
$$
A^T A + \beta \sum_{i=1}^m \sigma_i G_i^T G_i.
$$
A sufficient condition for this is:
\begin{equation}\label{eq: poss}
\sigma_i \geq 0 \quad \text{for } i = 1, \dots, m.
\end{equation}
In the next theorem, we show that this condition holds for the entire set $D$ under Assumption \ref{ass:phi1}:
\begin{theorem}\label{the: pos}
Let $\psi: \mathbb{R} \to \mathbb{R}_+$ be an edge-preserving potential function such that Assumption \ref{ass:phi1} holds, and let $\sigma$ be the augmented variable associated with it. Then $\sigma \geq 0$.
\end{theorem}

\begin{proof}
From Theorem \ref{th: peic} we know that:
$$
\psi(t) = V(t^2),
$$
and from Assumption \ref{ass:phi1}.1
 that the derivative  of $\psi$ must be non-negative, so  we can write:
\begin{equation}\label{eq: dvt}
\psi'(t)=2 t V'(\Phi(t))\ge0.
\end{equation}
In our application we have that $t=\|G_i x \|\ge 0$, therefore the (\ref{eq: dvt}) is verified if and only if:
\begin{equation}\label{eq: posdev}
V'(\Phi(t))\ge0.
\end{equation}
Finally from the theory of conjugate functions, it is known that the domain of the conjugate function $V^*$ is the range of the derivative of $V$, that from  (\ref{eq: posdev})  is a subset of $\real_+$, proving the theorem.
\end{proof}

From this theorem, it is proven that the function $L$ is convex in the variables $x$ for fixed $\sigma$, and convex in $\sigma$ for fixed $x$, resulting in a biconvex \cite{bicon} optimization problem over its entire domain. Biconvex functions are a well-known extension of convex functions, from which they inherit many favorable properties. 
As a matter of fact, a simple and efficient way to search for a minimum is to use a block coordinate strategy, which alternately solves the two convex subproblems until a solution is found. Furthermore, results in the literature prove the convergence of certain algorithms to the global solutions of biconvex functions \cite{flo90,flo99}. Further information and properties of biconvex functions can be found in the survey \cite{bicon}.

In Table \ref{tab:dualdomain} we report the most used edge preserving functions in literature \cite{nik07,te98}. In it possible to notice that for each  edge-preserving function the corresponding $V(y)$, with $y = \Phi(t) = t^2$ and $t = \|G x\|$ is concave, and the domain $D$ of the augmented variables is a subset of  $\real_+$, showing that both Theorems \ref{th: peic}  and \ref{the: pos} are indeed verified in practice.

From the discussion in this section, it emerges that Problems (\ref{eq: primep}) and (\ref{eq: lep}) can be considered {\it equivalent} in the sense that a solution to Problem (\ref{eq: primep}) is also a solution to Problem (\ref{eq: lep}), and vice versa. However, solving Problem (\ref{eq: lep}) is more advantageous from a computational perspective. As a matter of fact, exploiting the biconvex nature of $L_{ep}$ and solving a sequence of simple convex subproblems generally results in less computational time and better solutions than solving a non-convex optimization problem. 
The only drawback is the increased number of variables. However, because the subproblems are simpler and solved in sequence, solving Problem (\ref{eq: lep}) does not typically require more memory resources than solving Problem (\ref{eq: primep}).

\begin{table}[h]
\caption{Examples of edge-preserving potential functions and their Fenchel conjugates. Here $t = \|G x\|$ and $y = \Phi(t) = t^2$.}
\centering
\begin{tabular}{|c|c|c|c|c|}
\hline
$\psi(t)$ & $V(y)$ & $\nabla V(y)$ & $D$ & $V^*(\sigma)$ \\
\hline
$1 - \exp(- t^2)$ & $1 - \exp(- y)$ & $\exp(-y)$ & $\mathbb{R}_+$ & $-\sigma(\log \sigma - 1)$ \\
\hline
$\dfrac{ t^2}{1 +  t^2}$ & $\dfrac{ y}{1 +  y}$ & $\dfrac{1}{(1 + y)^2}$ & $[1, +\infty)$ & $\dfrac{(\sigma - \sqrt{\sigma})^2}{\sigma}$ \\
\hline
$\log( t^2)$ & $\log( y)$ & $\dfrac{1}{y}$ & $\mathbb{R}_{++}$ & $1 + \log \sigma$ \\
\hline
$\begin{cases}
\sin( t^2) & \text{if } 0 \le t \le \sqrt{\frac{\pi}{2}} \\
1 & \text{if } t > \sqrt{\frac{\pi}{2}}
\end{cases}$ &
$\sin( y)$ &
$\cos(y)$ &
$[0, 1]$ &
$\begin{array}{c}
\sigma \arccos(\sigma) - \\
\sin(\arccos(\sigma)) - 1
\end{array}$ \\
\hline
\end{tabular}

\label{tab:dualdomain}
\end{table}

\section{Conclusions}

In this work, we introduce a new class of functions, referred to as implicit concave functions, for which it is possible to obtain an augmented reformulation through the use of the Fenchel transformation. The minimization problem for both the original and the augmented functions is equivalent, in the sense that each minimum point of the original function corresponds to a minimum point in the augmented function, and vice versa.
However, minimizing the augmented function presents distinct advantages. Specifically, the augmented function is convex with respect to the newly introduced variables and is often simpler to minimize with respect to the original variables. This makes it particularly well-suited for block coordinate descent algorithms. Moreover, if there exists a subset of the augmented space where the augmented function is convex in the original variables, the theory and algorithms of biconvex optimization can be applied.

Implicit concave functions are already employed in machine learning, particularly in the form of half-quadratic regularization for edge-preserving functions. Using the theory developed in this paper, we demonstrate that half-quadratic regularization has a one-to-one correspondence with the stationary points of the original problem. Furthermore, we show that it is biconvex across its entire domain and bounded from below, indicating that solving the augmented problem offers almost no drawbacks compared to the original problem.

For future research, it would be of interest to explore other non-convex optimization problems that can be expressed in implicit concave form and to develop specialized algorithms for their corresponding augmented functions.

\section*{Bibliography}
\bibliographystyle{unsrt}
\bibliography{scholar}

\end{document}